\documentclass[12pt]{article}
\usepackage{e-jc}
\usepackage{amsthm,amssymb, amsmath,color}
\usepackage{amscd, url}
\usepackage{epsf, graphicx, epsfig, epstopdf}

\newtheorem{thm}{Theorem}[section]
\newtheorem{lem}[thm]{Lemma}
\newtheorem{prop}[thm]{Proposition}

\newtheorem{ques}[thm]{Question}
  \theoremstyle{definition}
\newtheorem{defi}[thm]{Definition}
\newtheorem{remark}[thm]{Remark}
\newtheorem{ex}[thm]{Example}

\newcommand{\R}{{\mathbb{R}}}

\newcommand{\bz}{{\bf Z}}

\title{Symmetric alcoved polytopes}
\author{Annette Werner\\
\small Institut f\"ur Mathematik\\
\small Goethe-Universit\"at Frankfurt\\
\small Frankfurt-am-Main, Germany\\
\small \tt werner@math.uni-frankfurt.de\\
\and
Josephine Yu \thanks{Supported by NSF grant \#1101289}\\
\small School of Mathematics\\ 
\small Georgia Institute of Technology\\ 
\small Atlanta GA, USA\\
\small \tt jyu@math.gatech.edu
}

\date{\dateline{Aug 18, 2013}{Dec 30, 2013}
}

\begin{document}

\maketitle

\begin{abstract}
Generalized alcoved polytopes are polytopes whose facet normals are roots in a given root system.  We call a set of points in an alcoved polytope a generating set if there does not exist a strictly smaller alcoved polytope containing it.  The type $A$ alcoved polytopes are precisely the tropical polytopes that are also convex in the usual sense. In this case the tropical generators form a generating set.  We show that for any root system other than $F_4$, every alcoved polytope invariant under the natural Weyl group action has a generating set of cardinality equal to the Coxeter number of the root system.
\end{abstract}



\section{Introduction}

In this paper we investigate alcoved polytopes which are symmetric under the action of the Weyl group.  For a root system $\Phi$, an alcoved polytope of type $\Phi$ is a polytope defined by inequalities of the form $\langle a , x \rangle \leq c$ where $a \in \Phi$ and $c\in \bz$.  They are unions of (faces of) alcoves in the affine Coxeter arrangement associated to $\Phi$.
Their combinatorics was studied by Lam and Postnikov \cite{lapo1, lapo2}, and Payne \cite{Pa} showed that  alcoved polytopes have a Koszul property.  Alcoved polytopes include important families of polytopes such as hypersimplices and generalized associahedra of Fomin and Zelevinsky \cite{fz}.  We consider them in a more general setting where we also allow non-integral coefficients.  Our main result clearly also holds for the original alcoved polytopes.

For type $A$ root systems, alcoved polytopes are precisely the tropical polytopes that are also convex in the usual sense.  They are named {\em polytropes} by Joswig and Kulas \cite{joku}.  Develin and Sturmfels showed that every tropical polytope has  a natural polyhedral cell decomposition into polytropes \cite{ds}.  It was shown in \cite{jsy} that tropical convexity is closely related to a notion of convexity in affine Bruhat--Tits buildings of type $A$.  In particular, the {\em membranes} in the building can be identified with tropical linear spaces, which are tropically convex.  It is natural to ask about the analogue of tropical convexity for buildings of other classical groups.  Although the right generalization of tropical convexity to other types has not yet been discovered, we believe that the alcoved polytopes and their generating sets will form important ingredients, as polytropes are the building blocks of tropical polytopes in type $A$.

Any $n$-dimensional alcoved polytope of type $A$ is the tropical convex hull of $n+1$ vertices \cite{joku}. We are interested in generating sets for alcoved polytopes associated to arbitrary root systems. Here a set of points in an alcoved polytope $P$ is called a generating set if $P$ is the smallest alcoved polytope containing it.  Our main result states that every alcoved polytope for an irreducible and reduced root system $\Phi$ not of type $F_4$ that is symmetric under the action of the Weyl group  can be generated by $h$ vertices, where $h$ is the Coxeter number of $\Phi$.  We also discuss the case $F_4$, where a symmetric alcoved polytope may need up to $24$ generators while the Coxeter number is $12$.

In the $A_n$-case however, \emph{every} alcoved polytope has a generating set of size $n+1$, which is the Coxeter number, regardless of symmetry.  The symmetry assumption cannot be dropped in general.  We show in Example~\ref{ex-D4} that there exists a non-symmetric alcoved polytope of type $D_4$ requiring eight generators, whereas the Coxeter number of $D_4$ is six. 

In Section~\ref{sec:tropical}, we review definitions of tropical convexity and alcoved polytopes and establish their relationship in the type $A$ case.  In Section~\ref{sec:gensets}, we define generating sets and give a short proof that every alcoved polytope of type $A_n$ has a generating set of $n+1$ elements.  Finally we prove our main result in Section~\ref{sec:symmetric}.

\section{Tropical convexity and alcoved polytopes}
\label{sec:tropical}
We begin by recalling some facts on tropical convexity. 
Let $(\R, \oplus, \odot)$ be the min-plus tropical semiring.
The space $\R^n$ together with componentwise addition $\oplus$ is a semimodule under the semiring $(\R, \oplus, \odot)$, if we put $a \odot (x_1,\ldots, x_n) = (a+ x_1, \ldots, a+ x_n)$ for $a \in \R$ and $(x_1,\ldots, x_n) \in \R^n$. 
\begin{defi}[\cite{ds}]
A subset $S$ of $\R^n$ is called {\em tropically convex} if for all $x, y $ in $S$ and for all $\mu, \nu \in \R$ the element $(\mu\odot x) \oplus (\nu \odot y)$ is also contained in $S$. The {\em tropical convex hull} of a subset $V$ in $\R^n$ is the smallest tropically convex set containing $V$. 
\end{defi}

Note that any tropically convex subset of $\R^n$ containing $x$ also contains $x + (r, \ldots, r) = r \odot x$ for all $r \in \R$. Therefore we project tropically convex subsets of $\R^n$ to $\R^n / \R (1, \ldots, 1)$ and call their image in $\R^n / \R (1, \ldots, 1)$ also tropically convex.

In \cite[Theorem 15]{ds} it is shown that the tropical convex hull of a finite subset of $\R^n / \R(1, \ldots, 1)$  admits a polyhedral cell decomposition.
These bounded cells are at the same time tropically convex and convex in the ordinary sense. Such a set is also called a {\em polytrope}; see \cite{joku}. 

Tropical convexity is closely related to  root systems of type $A$ in the following way. Let $U$ be the subspace $\{(x_1, \ldots, x_n): \sum_i x_i = 0 \}$ of $\R^n$. The quotient map \[ U \rightarrow \R^n / \R (1, \ldots, 1)\] is an isomorphism of real vector spaces. Hence the restriction of the canonical scalar product on $\R^n$ to $U$ induces a scalar product $\langle \,  , \, \rangle$ on  $ \R^n / \R (1, \ldots, 1)$. Consider the elements $a_{ij}= e_i - e_j$ in $\R^n / \R (1, \ldots, 1)$, where $e_i$ denotes the canonical unit vectors.  Then $\Phi = \{a_{ij}: i \neq j\}$ is a root system of type $A_{n-1}$ in $\R^n / \R (1, \ldots, 1)$. 

\begin{lem}
\label{lemma-An}  For any polytrope $P$ in $\R^n/\R(1, \ldots, 1)$ there are non-negative real constants $c_{ij}$ such that 
\[ P = \bigcap_{i \neq j} \{x \in \R^n/\R(1, \ldots, 1): \langle a_{ij}, x\rangle \leq c_{ij}\}.\]
\end{lem}
\begin{proof} By \cite[Lemma 10 and Proposition 18]{ds}, polytropes are precisely the bounded intersections of half-spaces of the form 
$\{ x= (x_1, \ldots, x_n) \in \R^n/ \R (1, \ldots, 1): x_i  - x_j \leq c_{ij}\}$ for $i \neq j$ and $c_{ij} \in \R$, which implies our claim.
\end{proof}

Hence polytropes look like alcoved polytopes of type $A$, as defined in \cite{lapo1,lapo2}. To be precise, 
let $V$ be a real vector space endowed with a scalar product $\langle\, , \, \rangle$, and let $\Phi$ be any irreducible and reduced root system in  $V$. By  $t$ we denote the type of the corresponding Dynkin diagram, i.e.\ $t$ is equal to $A_n$ for $n \geq 1$, $B_n$, $C_n$ for $n \geq 2$, $D_n$ for $n \geq 4$ or to $E_6$, $E_7$, $E_8$, $F_4$, or $G_2$. Moreover, let $W$ be the Weyl group of $V$, i.e.\ $W$ is the finite group of orthogonal endomorphisms of $V$ generated by the reflections at the hyperplanes $H_a = \{x \in V: \langle a,x \rangle = 0\}$ perpendicular to the roots.



\begin{defi} 
A {\em generalized alcoved polytope of type $t$} is a bounded intersection of root halfspaces of the form
\[\bigcap_{a \in \Psi} \{ x \in V: \langle a, x\rangle \leq c_a\},\]
where $\Psi$ is a subset of a root system $\Phi$ of type $t$ and $c_a \in \R$.
\end{defi}

Thus a generalized alcoved polytope is a bounded subset of $V$ defined as an intersection of root halfspaces. Alternatively, we can describe it as a polytope whose set of facet normal vectors is contained in $\Phi$. 
If the right hand side values $c_a$ are all integers, a generalized alcoved polytope we defined is an alcoved polytope from \cite{lapo1,lapo2}.  For simplicity we will refer to ``generalized alcoved polytopes'' simply as ``alcoved polytopes''.

Some examples of alcoved polytopes include associahedra (see \cite{joku} and references therein), cubes, cross polytopes, and the $24$-cell as a degenerate $F_4$ alcoved polytope.



\section{Generating sets of alcoved polytopes}
\label{sec:gensets}

We will define generating sets of alcoved polytopes and prove results about their cardinality for type $A$. Besides, we discuss an interesting example for type $D_4$. 

\begin{defi}\label{defi-generators} Let $\Phi$ be a root system of type $t$ in $V$, and let $P$ be an alcoved polytope of type $t$ in $V$. A subset $\Sigma \subset V$ is called a {\em generating set} of $P$ if $P$ is the smallest alcoved polytope in $V$ containing $\Sigma$. 
\end{defi}

Let $P$ be an alcoved polytope for the root system $\Phi$  of type $A_{n-1}$ as in Lemma~\ref{lemma-An}. Since alcoved polytopes are tropically convex, every subset of $P$ whose tropical convex hull is $P$ is a generating set. However not every set of generators of $P$ in the sense of Definition~\ref{defi-generators} is a set of tropical generators. 
The tropical convex hull of a generating set may be strictly smaller than $P$. For example, consider the square with vertices $(0,0,0), (0,0,1), (0,1,0), (0,1,1)$ in $\R^3/\R(1,1,1)$. The two vertices $(0,0,1)$ and $(0,1,0)$ generate the square as an alcoved polytope, but their tropical convex hall is strictly smaller.  In fact, the square cannot be generated tropically by two points, as the tropical convex hull of any two points is at most one dimensional in $\R^n / \R(1,\dots,1)$.

For any bounded $S \subseteq V$ and $a \in \Phi$ the smallest halfspace of the form $\{x \in V: \langle a, x \rangle \leq c\}$ that contains $S$ is $\{x \in V: \langle a,x \rangle \leq \sup_{s\in S}\langle a,s \rangle \}$, so
 the alcoved polytope generated by $S$ is
$$
\bigcap_{a \in \Phi}\{x \in V : \langle a, x \rangle \leq \sup_{s \in S} \langle a, s \rangle \}.
$$
On the other hand, for an alcoved polytope $P$, a subset $S \subset P$ is a generating set if and only if for every $a \in \Phi$ we have $$\sup_{x \in P} \langle a, x \rangle = \sup_{x \in S} \langle a, x \rangle.$$
This can be rephrased as follows.  Recall that a {\em support hyperplane} $H$ of a polytope $P$ is a hyperplane containing $P$ on one side and having non-empty intersection with $P$.  The intersection $H \cap P$ is the {\em face} of $P$ supported by $H$.

\begin{lem}\label{lem-generators-support}
A set $S \subseteq P$ is a generating set for $P$ if and only if for any $a \in \Phi$ the support hyperplane of $P$ perpendicular to $a$ contains at least one point from $S$.
\end{lem}

It follows that there is a minimal cardinality generating set consisting of vertices of $P$ only.  If a point in a generating set lies in the relative interior of a non-vertex face, then the point can be replaced by a vertex of that face, as doing so only enlarges the set of support hyperplanes containing the point.

\begin{ques}
Our problem can be phrased as follows.
\begin{enumerate}
\item For each root system $\Phi$, find a number $k$ such that every alcoved polytope of type $\Phi$ has a generating set of size at most $k$.
\item More generally, let $P$ be a polytope and $\mathcal{H}$ be a set of support hyperplanes of $P$ such that all facets of $P$ are contained in $\mathcal{H}$. The containment may be strict.  Find the minimal cardinality of a set $S \subset P$ such that  every support hyperplane in $\mathcal{H}$ contains at least one point from $S$.
\end{enumerate}
\end{ques}

There is a nice solution for type $A$ alcoved polytopes, as we will see in Proposition~\ref{prop-An-generators} below.
The general problem seems computationally difficult, as it amounts to solving a set covering problem, which is known to be NP-complete in general.  

\begin{remark}\label{rem-lowerbound} For every root system $\Phi$ there exist alcoved polytopes which cannot be generated with less than $ |\Phi| / \dim (V)$ vertices. For a fixed set of facet normals, the condition on the coefficients for the perpendicular hyperplanes to meet at a point is described by vanishing of certain determinants.  Generic choices of coefficients would avoid this determinantal locus and give polytopes that are simple, i.e.\ no more than $\dim(V)$ facets can meet at one point. Therefore in this case we need at least $|\Phi|/\dim(V)$ vertices to support all facets.
\end{remark}

\begin{prop} \label{prop-dim2} Let $t$ be  the type of a root system $\Phi$ in a vector space of dimension $2$. Then every alcoved polytope of type $t$ can be generated by $|\Phi| / 2$ many vertices. 
\end{prop}

\begin{proof}
Let us cyclically order the roots using their positions in the plane.  For any two support hyperplanes corresponding to two consecutive roots, there is a vertex supporting them, so $|\Phi|/2$ vertices are sufficient to support all root hyperplanes.
\end{proof}


In \cite[Theorem 12]{joku}, it is shown that every alcoved polytope of type $A_{n}$  is a tropical simplex, i.e.\ it has $d+1$ tropical generators, where $d$ is the dimension of the alcoved polytope (which might be smaller than $n$ if the polytope lies in a root hyperplane).
In particular it can be generated by $n+1$ elements in the sense of Definition~\ref{defi-generators}. 
The next proposition provides an easy proof of this fact.  This does not imply the Joswig--Kulas result because the set of tropical generators may be strictly larger.
\begin{prop}\label{prop-An-generators}
Every alcoved polytope of type $A_n$ is generated by $n+1$ elements.
\end{prop}

\begin{proof}
Let $P$ be a nonempty alcoved polytope of type $A_n$ defined by $x_i - x_j \leq c_{ij}$ for some $c_{ij} \in \R$ where $0 \leq i,j \leq n$ and $i \neq j$.  By Lemma~\ref{lem-generators-support}, we need to consider the support hyperplanes of $P$ normal to the roots, so we may assume that all inequalities are tight, i.e.\ for each pair $(i,j)$, there is a point $p \in P$ such that $p_i - p_j = c_{ij}$.  Then we must have $c_{ik} - c_{jk} \leq c_{ij}$ for any distinct $i,j,k$. Otherwise, for any $p \in P$, $p_i - p_k = (p_i - p_j) + (p_j - p_k) \leq c_{ij} + c_{jk} < c_{ik}$, contradicting our assumption that the inequality $x_i - x_k \leq c_{ik}$ is tight for $P$.

For each $0 \leq k \leq n$, let $p^{(k)} = (c_{0k}, \dots, c_{k-1,k}, 0, c_{k+1,k}, \dots, c_{nk}) \in \R^{n+1} / \R (1, \ldots, 1)$.  For any $l \neq k$, the point $p^{(k)}$ lies on the hyperplanes $x_\ell - x_k = c_{\ell k}$.  Since $P$ is non empty, let $p \in P$ and we get $-c_{ \ell k} \leq p_k - p_\ell \leq c_{k \ell }$.
Hence $p^{(k)}_k - p^{(k)}_\ell = - c_{\ell k} \leq c_{k \ell}$. 
Moreover, for any pair $i \neq j$ distinct from $k$, we have $p^{(k)}_i - p^{(k)}_j = c_{ik} - c_{jk} \leq c_{ij}$ using the inequality seen above.  Therefore the point $p^{(k)}$ lies in the polytope $P$.  The $n+1$ points $p^{(0)}, p^{(1)}, \dots, p^{(n)}$ all lie in $P$ and support all the root hyperplanes, so they form a generating set of size $n+1$.
\end{proof}

Note that $n+1$ is the Coxeter number of the root system of type $A_n$. We will see below in Theorem~\ref{theorem-generators} that almost all Weyl group symmetric alcoved polytopes can be generated by a set of vertices whose cardinality is the Coxeter number of the root system.

However, for types of root systems other than $A_n$ this is no longer true if we drop the symmetry hypothesis, as seen in the following example.



\begin{ex}
\label{ex-D4}
Consider the $D_4$ alcoved polytope in $\R^4$ defined by the inequalities given in the columns of the following matrix in homogeneous coordinates.  For example, the first column represents the inequality $100 + x_1 + x_2 \geq 0$.
$$\left(\begin{smallmatrix}{100}&
       {97}&
       {96}&
       {95}&
       {96}&
       {98}&
       {95}&
       {98}&
       {96}&
       {98}&
       {96}&
       {96}&
       {98}&
       {98}&
       {99}&
       {99}&
       {95}&
       {96}&
       {95}&
       {96}&
       {95}&
       {99}&
       {95}&
       {100}\\
       1&
       {-1}&
       1&
       {-1}&
       1&
       {-1}&
       1&
       {-1}&
       0&
       0&
       0&
       0&
       1&
       {-1}&
       1&
       {-1}&
       0&
       0&
       0&
       0&
       0&
       0&
       0&
       0\\
       1&
       {-1}&
       {-1}&
       1&
       0&
       0&
       0&
       0&
       1&
       {-1}&
       1&
       {-1}&
       0&
       0&
       0&
       0&
       1&
       {-1}&
       1&
       {-1}&
       0&
       0&
       0&
       0\\
       0&
       0&
       0&
       0&
       1&
       {-1}&
       {-1}&
       1&
       1&
       {-1}&
       {-1}&
       1&
       0&
       0&
       0&
       0&
       0&
       0&
       0&
       0&
       1&
       {-1}&
       1&
       {-1}\\
       0&
       0&
       0&
       0&
       0&
       0&
       0&
       0&
       0&
       0&
       0&
       0&
       1&
       {-1}&
       {-1}&
       1&
       1&
       {-1}&
       {-1}&
       1&
       1&
       {-1}&
       {-1}&
       1\\
       \end{smallmatrix}\right)$$
A computer check shows that this alcoved polytope requires eight generators, while the Coxeter number of $D_4$ is six.  The polytope is simple and has $f$-vector $(96, 192, 120, 24, 1)$.

The number of generators is equal to the cardinality of the smallest
subset of vertices that meets all the facets. We computed it as
follows. Using the software polymake \cite{polymake} compute the set of vertices in every facet.  With the software Macaulay2, we create a polynomial ring with variables indexed by the vertices.  To every facet, we associate a monomial that is the product of variables corresponding to the vertices in the facet.  These monomials generate an ideal whose codimension is the quantity we want to compute.
\end{ex}

\section{Symmetric alcoved polytopes}
\label{sec:symmetric}
In this section we define symmetric alcoved polytopes and prove our main result stating that in all cases except $F_4$ they can be generated by a set of cardinality equal to the Coxeter number. Afterwards we discuss the case of $F_4$-alcoved polytopes.

Let $\Phi$ be a root system contained in the Euclidean vector space $(V, \langle \, , \, \rangle)$ and let $W$ be the associated Weyl group.

\begin{defi}
We call an alcoved polytope in $V$ of type $t$ {\em symmetric} if it is invariant under the action of the Weyl group $W$.
\end{defi}

Note that $W$ acts on the faces of a symmetric alcoved polytope. Since every $w$ in $W$ maps the support hyperplane with respect to the root $a$ to the support hyperplane with respect to the root $w(a)$, we can write a symmetric alcoved polytope as
\[P = \bigcap_{a \in \Phi}\{ x \in V: \langle a, x\rangle \leq c_a\},\]
with $c_a = c_{w(a)}$ for all $w$ in $W$ and for all $a \in \Phi$.

Note that an alcoved polytope for a non-reduced root system $\Phi$ is also an alcoved polytope of the associated reduced root system $\Phi^{red}$. Moreover, alcoved polytopes for reducible root systems are products of alcoved polytopes for the irreducible components. Therefore we will from now on assume that $\Phi$ is irreducible and reduced. 

Let us recall some basic facts about root systems from \cite[Chapter VI]{bou}.  For every root $a$ we denote by $s_a$ the corresponding reflection in $V$. Then for all roots $b$ we have
\[s_a(b) = a - n(a,b) b, \quad \mbox{ where }n(a,b) = 2 \frac{\langle a,b\rangle}{\langle b,b\rangle}\]
is a number between $-3$ and $3$ in the reduced case. 

Fix a chamber with respect to the hyperplane arrangement in $V$ given by the hyperplanes $H_a =\{ x \in V: \langle a, x \rangle = 0 \}$. Let $b_1, \ldots, b_n$ be the roots perpendicular to the facets of the fixed chamber which lie on the same side as the chamber.  They form a basis of $\Phi$ \cite[VI, 1.5, Theorem 2]{bou}.  Then the element
\[ \omega = s_{b_1} \circ \ldots \circ s_{b_n}\]
of the Weyl group is called a {\em Coxeter element}. It depends up to conjugation on the choice of the basis and on the choice of its ordering \cite[V, 6.1, Proposition 1]{bou}. The order $h$ of $\omega$ in the Weyl group is called the {\em Coxeter number} of $\Phi$. By \cite[V, 6.2, Theorem 1]{bou},  we have
\[n h = |\Phi|\]
where $|\Phi|$ is the number of roots. 

After choosing a basis $b_1, \ldots, b_n$ we also write $s_i = s_{b_i}$.  Let $\Gamma$ be the cyclic subgroup of $W$ generated by $\omega = s_1 \circ \ldots \circ s_n$. 
For all $i = 1, \ldots, n$ let
\[\theta_i = s_n \circ s_{n-1} \circ \ldots \circ s_{i+1} (b_i).\]

\begin{prop}\cite[VI, 1.11, Proposition 33]{bou}
With the notation above, the root system $\Phi$ is a disjoint union of the orbits $\Gamma \theta_i$ for $i = 1, \ldots n$.
\end{prop}

Our main result is as follows.
\begin{thm}\label{theorem-generators} Let $P$ be a symmetric alcoved polytope of irreducible and reduced type $t \neq F_4$. Then $P$ can be generated by $h$ vertices, where $h$ is the Coxeter number of $t$.
\end{thm}

\begin{proof}
If  $\dim V = 2$, our claim follows from Lemma~\ref{prop-dim2}. The simply laced case is proved below in Theorem~\ref{theorem-simplylaced}. The case
$B_n$ and $C_n$ will be shown in Theorem~\ref{theorem-bc}.
\end{proof}

Symmetric alcoved polytopes of type $F_4$ may require up to $24$ generators while the Coxeter number is $12$; see Theorem~\ref{theorem-f4} for a detailed analysis.

Note that by Remark~\ref{rem-lowerbound} the bound from Theorem~\ref{theorem-generators} is the candidate for the best possible result for general (not necessarily symmetric) alcoved polytopes, since  $h=| \Phi | / \dim (V)$.  This works in the $A_n$ case, as explained above. However, in the case of the root system $D_4$ there are exceptions of non-symmetric alcoved polytopes where more generators are needed; see Example~\ref{ex-D4}. It is an interesting question for which types of root system the bound from Theorem~\ref{theorem-generators} can be generalized to the non-symmetric case. 

\subsection{Simply laced case: types $A$, $D$, and $E$.}

As indicated before, we will prove Theorem~\ref{theorem-generators} in several steps. We begin with the simply laced case.
Recall that a root system $\Phi$ is called simply laced if all roots have the same length, i.e.\ if the number $\langle a,a \rangle$ is constant for all $a \in \Phi$. 
The types of the simply laced root systems are precisely the types $A,D$ and $E$. 

\begin{thm}\label{theorem-simplylaced} Let $P$ be a symmetric alcoved polytope of type $t$, where $t$ is irreducible, reduced and simply laced, i.e.\ either $A_n$ for $n \geq 1$, $D_n$ for $n \geq 4$ or $E_6$, $E_7$ or $E_8$. Then $P$ can be generated by $h$ vertices.
\end{thm}

\begin{proof} We assume first that 
 $t$ is not $E_8$.
We choose a basis $b_1, \ldots, b_n$ of $\Phi$ and look at the corresponding Dynkin diagram.  We number the basis elements in the following way:
\begin{center}
\includegraphics[angle=0, width=0.6\textwidth]{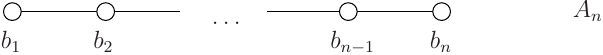}
\end{center}
\begin{center}
\includegraphics[angle=0,width=0.6\textwidth]{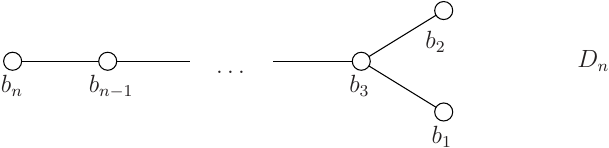}
\end{center}
\begin{center}
\includegraphics[angle=0,width=0.6\textwidth]{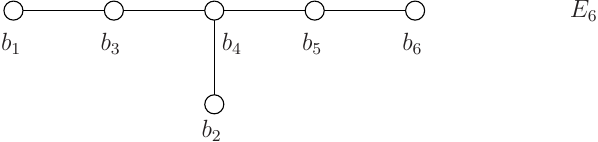}
\end{center}
\begin{center}
\includegraphics[angle=0,width=0.6\textwidth]{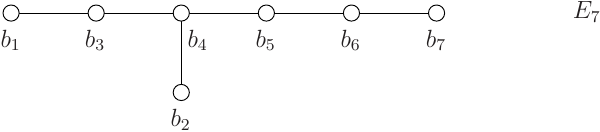}
\end{center}
Note that the numbering in the $D_n$ case is opposite to the numbering often found in the literature.
Recall that every vertex in the Dynkin diagram represents a basis element $b_i$ and that $\langle b_i,b_j\rangle$ is $-\langle b_j , b_j \rangle$ if the vertices for $b_i$ and $b_j$ are connected, and $0$, if they are not. Since the Dynkin diagram is a tree, for all $i$ and $j$ there exists a geodesic $[b_i, b_j]$ connecting the vertices $b_i$ and $b_j$. 
We first show the following claim. 

{\it Claim: Let $i\geq 1$, if $t$ is $A_n$, and $i \geq 2$, if $t$ is $D_n$, and  $i\geq 3$, if $t$ is $E_6$ or $E_7$. Then for all $j \geq i$ we have
\[s_j \circ \ldots \circ  s_{i+1}(b_i) = \sum_{\nu = i}^j b_\nu.\]}

Let us prove the claim by induction on $j$. 
If $j$ is equal to $i$ there is nothing to prove. Assume that the claim holds for some $j \geq i$ with $j <n$. Then we calculate
\begin{eqnarray*}s_{j+1} \circ s_j \circ \ldots s_{i+1}(b_i) 
&  = &s_{j+1}( \sum_{\nu = i}^j b_\nu) \\
& =  & \sum_{\nu = i}^{j-1} b_\nu + s_{j+1}(b_j) \\
& = & \sum_{\nu = i}^{j-1} b_\nu + b_{j} + b_{j+1} \\
& = & \sum_{\nu = i}^{j+1} b_\nu,
\end{eqnarray*}
which finishes the proof of the claim.

Recall that for all $i = 1, \ldots n$ we have
\[  \theta_i = s_n \circ \ldots \circ s_{i+1} (b_i).\]
From the claim we deduce that $\quad \theta_i = \sum_{\nu = i}^n b_\nu$, which implies
\[(\ast) \quad \theta_i = \sum_{b_\nu \in [b_i,b_n]} b_\nu\]
except possibly in the cases $t=D_n$ and $i=1$ or
$t = E_6$ or $E_7$ and $i =1,2$. 
It can easily be checked directly that the statement $(\ast)$ 
also holds in these remaining cases. 

Since in all cases $b_n$ is a neighbor of $b_{n-1}$ in the Dynkin diagram,
$(\ast)$ implies that 
\[\theta_i = b_n + \sum_{b_\nu \in [b_i,b_{n-1}]} b_\nu.\]

In a simply laced root system, the Weyl group acts transitively on the set
of roots, see \cite[VI, 1.3, Proposition 11]{bou}. Hence every 
symmetric alcoved polytope is of the form 
\[P_{\lambda} = \bigcap_{a \in \Phi} \{x : \langle a, x\rangle \leq \lambda\}\]
for some $\lambda> 0$.

Since $b_1, \ldots, b_n$ is a basis of $V$, there is a unique point $x \in V$ such that 
\[\langle b_n, x \rangle = \lambda, \langle b_{n-1}, x \rangle = \ldots = \langle b_2, x \rangle =\langle b_1, x \rangle = 0.\]
We will show now that $x$ is a vertex of $P_\lambda$. 

Note that the highest root in $\Phi$ is of the following form:
\[\begin{array}{ll}
A_n:& b_1 + \cdots + b_n\\
D_n : & b_1 + b_2 + 2 b_3 + \cdots + 2 b_{n-1} + b_n \\
E_6 : & b_1 + 2b_2 + 3 b_3 + 3 b_4 + 2 b_5 +  b_6 \\
E_7: & 2 b_1 + 2 b_2 + 3 b_3 + 4 b_4 + 3 b_5 + 2 b_6 + b_7
\end{array}
\]
Therefore for every positive root $a \in \Phi$ we find that $\langle a, x \rangle = 0$ or $\langle a,x \rangle = \langle b_n, x\rangle = \lambda$. For the negative roots there is nothing to do, since they assume a non-positive value on $x$. Hence we find that $x$ lies in $P_\lambda$. Since $\theta_1, \ldots, \theta_n$ are independent (see \cite{bou}, VI, 1.11, Proposition 33), and $x$ lies in all hyperplanes $\{\langle \theta_i, - \rangle  = \lambda\}$, we find that $x$ is a vertex. 

Since the Weyl group maps faces to faces, the elements $\omega^k x$ for $k = 0, \ldots, h-1$ are also vertices of $P_\lambda$, where $\omega = s_{b_1} \circ \ldots \circ s_{b_n}$ is the corresponding Coxeter element. Since $x$ satisfies $\langle \theta_i, x\rangle = \lambda$ for all $i = 1, \ldots, n$, the vertex $ y = \omega^k (x)$ satisfies 
$\langle \omega^{k} \theta_i, y \rangle = \lambda$ for all $i = 1, \dots, n$. Since every root is of the form  $\omega^{k} \theta_i$ 
for suitable $i$ and $k$, every hyperplane appearing at the boundary of  $P_\lambda$ contains one of the vertices $\omega^k x$ for $k = 0, \ldots, h-1$. This implies our claim for all types except possibly $E_8$.

We conclude the proof by discussing the type $E_8$. Let
$$B = \left(\begin{smallmatrix}1/2&
     1&
     {-1}&
     0&
     0&
     0&
     0&
     0\\
     {-1/2}&
     1&
     1&
     {-1}&
     0&
     0&
     0&
     0\\
     {-1/2}&
     0&
     0&
     1&
     {-1}&
     0&
     0&
     0\\
     {-1/2}&
     0&
     0&
     0&
     1&
     {-1}&
     0&
     0\\
     {-1/2}&
     0&
     0&
     0&
     0&
     1&
     {-1}&
     0\\
     {-1/2}&
     0&
     0&
     0&
     0&
     0&
     1&
     {-1}\\
     {-1/2}&
     0&
     0&
     0&
     0&
     0&
     0&
     1\\
     1/2&
     0&
     0&
     0&
     0&
     0&
     0&
     0\\
     \end{smallmatrix}\right), ~ 
\omega = \left(\begin{smallmatrix}0&
     1&
     0&
     0&
     0&
     0&
     0&
     {-1}\\
     0&
     0&
     1&
     0&
     0&
     0&
     0&
     {-1}\\
     1&
     1&
     0&
     0&
     0&
     0&
     0&
     {-1}\\
     0&
     1&
     1&
     0&
     0&
     0&
     0&
     {-1}\\
     0&
     0&
     0&
     1&
     0&
     0&
     0&
     {-1}\\
     0&
     0&
     0&
     0&
     1&
     0&
     0&
     {-1}\\
     0&
     0&
     0&
     0&
     0&
     1&
     0&
     {-1}\\
     0&
     0&
     0&
     0&
     0&
     0&
     1&
     {-1}\\
     \end{smallmatrix}\right),  \Theta = \left(\begin{smallmatrix}1&
      0&
      0&
      0&
      0&
      0&
      0&
      0\\
      0&
      1&
      0&
      0&
      0&
      0&
      0&
      0\\
      1&
      0&
      1&
      0&
      0&
      0&
      0&
      0\\
      1&
      1&
      1&
      1&
      0&
      0&
      0&
      0\\
      1&
      1&
      1&
      1&
      1&
      0&
      0&
      0\\
      1&
      1&
      1&
      1&
      1&
      1&
      0&
      0\\
      1&
      1&
      1&
      1&
      1&
      1&
      1&
      0\\
      1&
      1&
      1&
      1&
      1&
      1&
      1&
      1\\
      \end{smallmatrix}\right).$$
Then the columns $b_1, b_2, \dots, b_8$ of the matrix $B$ are a basis for a root system of type $E_8$. It has the following Dynkin diagram

\begin{center}
\includegraphics[angle=0,width=0.6\textwidth]{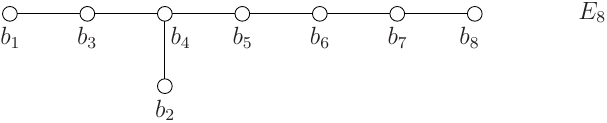}
\end{center}

The matrix $\omega$ represents the Coxeter element associated to $b_1, \ldots, b_8$, and the columns of $\Theta$ are the roots $\theta_i = s_n \circ \ldots \circ s_{i+1}(b_i)$ expressed in the basis $b_1, \ldots, b_8$.  Let $x = \frac{1}{2}(b_2+b_3)$ in $V$.  We verify by explicit calculation that the point $x$ lies on the root hyperplanes corresponding to $\omega^{17}\theta_1$,
$\omega^{20}\theta_1$,
$\omega^{16}\theta_2$,
$\omega^{18}\theta_2$,
$\omega^{17}\theta_3$,
$\omega^{15}\theta_4$,
$\omega^{20}\theta_4$,
$\omega^{9}\theta_5$,
$\omega^{14}\theta_5$,
$\omega^{13}\theta_6$,
$\omega^{12}\theta_7$,
$\omega^{11}\theta_8$,
$\omega^{20}\theta_8$, and
$\omega^{25}\theta_8$.
Hence $x$ touches all orbits of the root hyperplanes under the Coxeter element $\omega$, so the orbit of $x$ forms a $30$-element generating set for the symmetric alcoved polytope $P_1$ of $E_8$.  Modulo scaling, there is only one symmetric alcoved polytope  of type $E_8$, which has $f$-vector $(19440, 207360, 483840, 483840, 241920, 60480, 6720, 240, 1)$, computed using polymake.
\end{proof}

\subsection{Types $B$ and $C$}

  The $C_n$ roots can be obtained by rescaling some of the $B_n$ roots, so the $B_n$ and $C_n$ alcoved polytopes have the same facet directions.  Moreover $B_n$ and $C_n$ have the same Weyl group. Therefore the symmetric alcoved polytopes of type $C_n$ coincide with the symmetric alcoved polytopes of type $B_n$, and it suffices to consider the $B_n$ case.

Here we may assume $V = \R^n$ and $\Phi = \{\pm a_i: i = 1, \ldots n\} \cup \{\pm a_i \pm a_j: 1 \leq i < j \leq n\}$, where $a_i$ is  the canonical basis of $V$.  The Coxeter number for types $B_n$ and $C_n$ is $2n$.

\begin{thm}
\label{theorem-bc}
Let $P$ be a symmetric alcoved polytope of type $B_n$ or $C_n$ for $n \geq 2$. Then $P$ can be generated by $2n$ vertices.
\end{thm}

\begin{proof}

We choose as a basis $b_1 = a_1 - a_2, b_2 = a_2 - a_3, \ldots b_{n-1} = a_{n-1} - a_n$ and $b_n = a_n$ of $\Phi$. 
The Dynkin diagram of type $B_n$ is the following:
\begin{center}
\includegraphics[angle=0,width=0.6\textwidth]{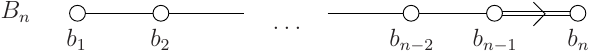}
\end{center}

We then calculate the  roots $\theta_1, \ldots, \theta_n$ as
\[\theta_n = a_n \mbox{ and }\theta_i = a_n + a_i \mbox{ for }i=1, \ldots, n-1.\]
Recall that the union of the orbits of the $\theta_i$ under the group $\Gamma$ generated by the Coxeter element $ \omega = s_1 \circ \ldots \circ s_n$ is the whole root system. Since Weyl group elements map long roots to long roots and short roots to short roots, we deduce that the set of short roots is equal to $\Gamma b_n$, and  that the set of long roots is equal to $\Gamma b_1 \cup \ldots \cup \Gamma b_{n-1}$. 







A type $B_n$ or $C_n$ symmetric alcoved polytope $P$ is of the form
\[P_{\lambda,\mu}= \bigcap_i \{x : \langle \pm a_i, x \rangle \leq \mu \} \cap \bigcap_{i<j} \{x: \langle (\pm a_i \pm a_j), x \rangle \leq \lambda\} \]
for positive numbers $\lambda$ and $\mu$ satisfying $\mu \leq \lambda \leq 2 \mu $. In the extreme cases $\mu = \lambda$ and $\lambda = 2 \mu$, only long or short root hyperplanes appear as facets, and the others support lower dimensional faces.  However, this would not affect our reasoning.
Let $x$ be the point in $V$ defined by 
\[\langle a_1, x \rangle  = \ldots = \langle a_{n-1}, x \rangle = \lambda - \mu \mbox{ and }\langle a_n, x \rangle = \mu.\]
Then $x$ is a point in $P_{\lambda, \mu}$ lying in the $n$ independent hyperplanes $\{\langle \theta_n, - \rangle  = \mu\}$ and $\{\langle \theta_i , - \rangle = \lambda\}$ for $i = 1, \ldots, n-1$. Therefore  $x$ is a vertex. 
As in the proof of Theorem~\ref{theorem-simplylaced}, the $h$ vertices $\omega^k x$ for $k = 0, \ldots, h-1$ touch every boundary hyperplane. Hence they form a generating set of cardinality $h$.
\end{proof}

\subsection{Type $F_4$}

Let $V = \R^4$ with canonical basis $a_1, \ldots, a_4$.  Then 
\[\Phi = \{\pm a_i: i=1, \ldots, 4\} \cup \{\pm a_i \pm a_j: 1 \leq i < j \leq 4\} \cup \{\frac{1}{2} (\pm a_1 \pm a_2 \pm a_3 \pm a_4)\}\]
is a root system of type $F_4$ in $V$. It consists of $24$ long roots (of the form $\pm a_i \pm a_j$) and $24$ short roots.

A symmetric alcoved polytope of type $F_4$ is of the form
\[P_{\lambda,\mu}= \bigcap_i \{ \langle \pm a_i, x \rangle \leq \mu \}
\cap \{ \langle(\pm a_1 \pm a_2 \pm a_3 \pm a_4), x \rangle \leq 2 \mu\} \cap \bigcap_{i<j} \{\langle (\pm a_i \pm a_j), x \rangle \leq \lambda\} \]
for positive numbers $\lambda$ and $\mu$ satisfying $\mu \leq \lambda \leq 2 \mu $. This condition ensures that all root hyperplanes are support hyperplanes. When $\mu = \lambda$ or $\lambda = 2 \mu$, either only long roots or only short roots appear as facet normals and the others support faces of lower dimension.  When $\mu < \lambda < 2 \mu $, all root hyperplanes support facets.

The convex hull of short roots of $F_4$ is a regular polytope called the {\em $24$-cell}.  The long roots are the vertices of its dual.  The $24$-cell is combinatorially equivalent to $P_{1,1}$.  The root system of type $F_4$ is self-dual. By doubling the lengths of the short roots, and then rescaling everything by $\frac{1}{\sqrt{2}}$, we get back $F_4$, rotated, with the roles of long and short roots switched.  More concretely, the matrix $A = \frac{1}{2}
\left(\begin{smallmatrix}
1 & 1 & 0 & 0 \\
1 & -1 & 0 & 0 \\
0 & 0 & 1 & 1 \\
0 & 0 & 1 & -1 
\end{smallmatrix}\right)
$
 maps the long roots to short roots and the doubled short roots to long roots. Then $A^{-T}$ is a linear isomorphism from $P_{\lambda, \mu}$ to $P_{2 \mu, \lambda}$, so these two polytopes have exactly the same combinatorics, in particular, the same number of generators.

\begin{thm}
\label{theorem-f4}
Let $P_{\lambda, \mu}$ be an alcoved polytope of type $F_4$, using the notation above.  
\begin{enumerate}
\item If $\frac{4}{3} \mu \leq \lambda \leq \frac{3}{2} \mu$, then $P_{\lambda, \mu}$ is generated by $12$ elements.  
\item If $\mu \leq \lambda \leq \frac{4}{3} \mu$ or $\frac{3}{2} \mu \leq \lambda \leq 2 \mu$, then $P_{\lambda, \mu}$ is generated by $24$ elements.  
\end{enumerate}
In each of these cases, if the inequalities among the $\lambda$ and $\mu$ are strict, then the numbers of generators, 12 and 24 respectively, are minimum possible.
\end{thm}

\begin{proof} 

The Dynkin diagram of a root system of type $F_4$ is the following:
\begin{center}
\includegraphics[angle=0,width=0.6\textwidth]{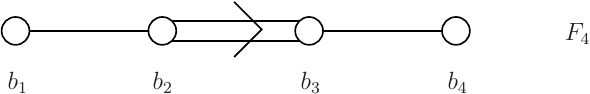}
\end{center}
with the basis $b_1 = a_2 - a_3, b_2 = a_3-a_4, b_3 = a_4$ and $b_4 = \frac{1}{2} (a_1-a_2-a_3-a_4)$ of $\Phi$. We calculate the roots $\theta_1, \ldots, \theta_4$ associated to the basis $\{b_1, b_2, b_3, b_4\}$:
\begin{eqnarray*}
\theta_1 & = &  a_1 - a_3\\
\theta_2  & = & a_1 - a_2 \\
\theta_3  & = & \frac{1}{2} (a_1 -  a_2 - a_3 + a_4) \\
\theta_4 & = & \frac{1}{2} (a_1 - a_2 - a_3 - a_4)
\end{eqnarray*}







\noindent {\bf Case (1)}: The alcoved polytope is of the form $P_{\lambda, \mu}$ with 
$\frac{4}{3} \mu  \leq \lambda \leq \frac{3}{2} \mu$.
Let $x$ be the element of $V$ given by 
\[ \langle a_1, x \rangle = 2 \lambda - 2 \mu,
\langle a_2, x \rangle =  \lambda - 2\mu, \langle a_3, x \rangle = \lambda - 2 \mu,  \langle a_4, x \rangle =  0.\]
We check that $x$ is a point in $P_{\lambda,\mu}$ as follows.
Using $\mu < \lambda$ and $2 \lambda \leq 3 \mu$, we find \[- \mu \leq 2 \lambda- 2 \mu \leq \mu \mbox{ and } -\mu \leq \lambda - 2 \mu \leq \mu.\]
 Hence $x$ lies in $\bigcap_i \{ \langle \pm a_i, x \rangle \leq \mu \}$.   Moreover, we have 
\[\pm (2 \lambda - 2 \mu) \pm (\lambda - 2 \mu) \pm (\lambda - 2 \mu) \leq 2 \lambda - 2 \mu + 2 ( 2 \mu - \lambda) = 2 \mu,\]
so $x$ lies in $\{ \langle(\pm a_1 \pm a_2 \pm a_3 \pm a_4), x \rangle \leq 2 \mu\}$.  Checking the long roots we find
\[ \pm (2 \lambda - 2 \mu) \pm (\lambda - 2 \mu) \leq 2 \lambda - 2 \mu + 2 \mu - \lambda = \lambda \textup{, and}\]
\[ - \lambda \leq 2\lambda - 2 \mu \leq \lambda\mbox{ and }- \lambda \leq \lambda - 2 \mu \leq \lambda.\]
It remains to check 
\[ \pm (\lambda - 2 \mu) \pm (\lambda - 2 \mu)  \leq 4 \mu - 2\lambda \leq \lambda,\]
which follows from the condition $4 \mu \leq 3 \lambda$.


Then we calculate $\langle \theta_1, x \rangle = \lambda, \langle \theta_2,x \rangle = \lambda, \langle \theta_3, x \rangle = \mu , \langle \theta_4, x \rangle = \mu $.
Hence $x$ lies in four linearly independent facets of $P_{\lambda, \mu}$, so it is a vertex. By the same argument as in the proof of Theorem~\ref{theorem-simplylaced}, the vertices $\omega^k x$ for $k= 0, \ldots, h-1$ generate $P_{\lambda,\mu}$. 

A computation using Gfan \cite{gfan} shows that all values of $\lambda$ and $\mu$ satisfying $\frac{4}{3} \mu  < \lambda < \frac{3}{2} \mu$ give the same combinatorial type for $P_{\lambda, \mu}$.  A polymake computation shows that the polytope is simple in this case, i.e.\ each vertex lies in exactly $4$ facets.  Hence at least $12$ vertices are necessary to meet all $48$ facets.  The polytope has $f$-vector $(288, 576, 336, 48,1)$.

\medskip

\noindent {\bf Case (2)}:  The remaining cases are $\mu \leq \lambda < \frac{4}{3} \mu$ and $\frac{3}{2} \mu < \lambda \leq 2 \mu$.  However, these two cases are equivalent via the isomorphism of $P_{\lambda, \mu}$ and $P_{2 \mu, \lambda}$ seen above.  

Let us consider the case $\frac{3}{2} \mu < \lambda \leq 2 \mu$.  The point $x_1 = (\mu, \mu - \lambda, \lambda - 2\mu, 0)$ lies on the facets perpendicular to $\theta_2$, $\theta_3$, and $\theta_4$, while the point $x_2 = (\mu,  \lambda - 2\mu, \mu - \lambda, 0)$ lies on the facets perpendicular to $\theta_1$, $\theta_3$, and $\theta_4$.  They also lie in the hyperplane defined by $\langle a_1, x \rangle = \mu$, so they are vertices of the polytope.  Since $x_1, x_2 \in P_{\lambda, \mu}$ and together they touch all  hyperplanes perpendicular to $\theta_i$'s, the union of the orbits of $x_1$ and $x_2$ under the group $\Gamma$ forms a generating set of size at most $24$.  

Computation by Gfan shows that all values of $\lambda$ and $\mu$ satisfying $\frac{3}{2} \mu < \lambda < 2 \mu$ give the same the combinatorial type for $P_{\lambda, \mu}$, which is simple and has $f$-vector $(192, 384, 240, 48, 1)$.  The $192$ vertices of $P_{\lambda, \mu}$ are obtained from $x_1$ (or equivalently $x_2$) by permuting coordinates in any of the $24$ ways and changing signs in any of the $8$ ways, for a total of $192$ distinct combinations.  It can then be seen directly that each vertex lies in exactly one long-root facet and there are $24$ long-root facets, so $P_{\lambda,\mu}$ needs at least $24$ generators.
\end{proof}

\end{document}